\documentclass{article}
\usepackage{amsmath,amsfonts,amsthm,comment,url}
\usepackage{graphicx}
\usepackage{hyperref}

\theoremstyle{plain}
\newtheorem{theorem}{Theorem}

\theoremstyle{definition}
\newtheorem{definition}[theorem]{Definition}
\newtheorem{example}[theorem]{Example}

\newcommand{\ff}[1]{{\underline{#1}}}

\newcommand{\Po}{\operatorname{Po}}
\newcommand{\Multi}{\operatorname{Multi}}
\newcommand{\NegMulti}{\operatorname{NegMulti}}
\newcommand{\Bin}{\operatorname{Bin}}

\newcommand{\NegBin}{\operatorname{NegBin}}
\newcommand{\Disp}{\operatorname{Disp}}
\newcommand{\Var}{\operatorname{Var}}
\newcommand{\DispM}{\operatorname{\mathsf{Disp}}}
\newcommand{\VarM}{\operatorname{\mathsf{Var}}}
\newcommand{\Cov}{\operatorname{Cov}}
\newcommand{\NN}{\operatorname{N}}
\newcommand{\Herm}{\operatorname{Herm}}

\newcommand{\multiset}[2]{\bigg(\kern-0.4em\dbinom{#1}{#2}\kern-0.4em\bigg)}
\newcommand{\tmultiset}[2]{\big(\kern-0.2em\tbinom{#1}{#2}\kern-0.2em\big)}

\begin{document}

\title{Marking and re-marking}
\author{Matthew Aldridge}

\maketitle

\begin{abstract}
A random number of items each independently marked with one of a collection of colours gives rise to the multinomial marking, which generalises binomial thinning. A multivariate version, where previously marked items are then re-marked, has similar properties to taking a linear transformation of a random vector.
\end{abstract}

\section{Multinomial marking} \label{sec:marking}

\subsection{Introduction and definitions} \label{sec:intro}

Suppose I have a random number $X$ of balls. Each ball is independently painted red with probability $r \in (0,1)$ or blue with probability $b = 1 - r$. Let $R$ be the number of red balls and $B$ the number of blue balls. Are $R$ and $B$ positively correlated, negatively correlated, or uncorrelated? Every ball painted red is a missed opportunity to paint it blue, and vice versa, which might suggest negative correlation. But the original number of unpainted balls was random, so if we see lots of red balls, that suggests there were probably lots of unpainted balls to start with, so there should be lots of blue balls too, which might suggest positive correlation. Which is true?

Answering this question requires studying \emph{marking}. Suppose we have a fixed number $n$ of objects, each of which is marked by one of $c$ colours: colour $1$ with probability $a_1$, colour $2$ with probability $a_2$, and so on. If $Y_i$ is the number of objects marked by colour $i$, then $\mathbf Y = (Y_i)$ is said to follow the \emph{multinomial distribution}. If the number of objects is not a fixed number $n$ but rather a random number $X$, then we call $\mathbf Y$ a \emph{multinomial marking} of $X$. 

In Section \ref{sec:marking} we study this multinomial marking. In this subsection, Subsection \ref{sec:intro}, we formally define marking; in Subsection \ref{sec:ex} we see some examples of markings of probability distributions; then in Subsection \ref{sec:prop} we examine some of the properties of marking and answer our opening question about red and blue balls. The concept of marking can be extended to taking some already marked objects and then giving them a new `re-marking'. We study this re-marking in Section \ref{sec:remarking}, where we see that this has similar properties to taking a linear transformation $\mathsf A \mathbf X$ of a real-valued random variable $\mathbf X$.

Multinomial marking can be considered as a multivariate extension of binomial thinning. In \emph{thinning}, each item is independently kept with probability $a$ or discarded with probability $1 - a$. 
Work that has considered thinning of random variables includes \cite{renyi,rao,steutel1,HJK,townes,aldridge} and much more.

Some notation. The non-negative integers are $\mathbb N = \{0, 1, 2, \dots\}$. For a vector $\mathbf y = (y_1, y_2, \dots, y_c) \in \mathbb R^c$, we write $|\mathbf y| = \sum_{i=1}^c y_i$ for the weight of $\mathbf y$. We use Knuth's notation
\[ x^{\ff{k}} = x(x-1) \cdots (x-k+1) \]
for the falling factorial. The binomial and multinomial coefficients are
\[ \binom{n}{k} = \frac{n^{\ff{k}}}{k!} \qquad \binom{n}{\mathbf k} = \frac{n^{\ff{|\mathbf k|}}}{\prod_{i=1}^c k_i!} .\]

\begin{definition}
Let $n \in \mathbb N$ and $p \in [0,1]$. A one-dimensional random variable $Y$ taking values in $\mathbb N$ with probability mass function
\[ p(y) = \binom{n}{y} p^{y} (1-p)^{n-y}  \]
is said to be a \emph{binomial distribution} with parameters $n$ and $p$. We write $Y \sim \Bin(n, p)$.

Let $n \in \mathbb N$ and $\mathbf p \in [0,1]^c$ with $|\mathbf p| \leq 1$. A $c$-dimensional random variable $\mathbf Y$ taking values in $\mathbb N^c$ with probability mass function
\[ p(\mathbf y) = \binom{n}{\mathbf y} \prod_{i=1}^c p_i^{y_i} \, \big(1 - |\mathbf p|\big)^{n-| \mathbf y|}  \]
is said to be a \emph{multinomial distribution} with parameters $n$ and $\mathbf p$. We write $\mathbf Y \sim \Multi(n, \mathbf p)$.

Let $X$ be a random variable taking values in $\mathbb N$ and $a \in [0,1]$. Consider the one-dimensional random variable $Y$ taking values in $\mathbb N$ whose conditional distribution given $X$ is $\Bin(X,a)$, so its probability mass function is
\[ p(y) = \mathbb E \left(\binom{X}{y} a^{y} (1-a)^{X-y}\right) . \]
We call $Y$ the \emph{binomial thinning} (or just \emph{thinning}) of $X$ with parameter $a$, and write $Y = a \circ X$.

Let $X$ be a random variable taking values in $\mathbb N$ and $\mathbf a \in [0,1]^c$ with $|\mathbf a| \leq 1$.  Consider the $c$-dimensional random variable $\mathbf Y$ taking values in $\mathbb N^c$ whose conditional distribution given $X$ is $\Multi(X,\mathbf a)$, so its probability mass function is
\[ p(\mathbf y) = \mathbb E \left(\binom{X}{\mathbf y} \prod_{i=1}^c a_i^{y_i} \, \big(1 - |\mathbf a|\big)^{X-| \mathbf y|} \right) . \]
We call $\mathbf Y$ the \emph{multinomial marking} (or just \emph{marking}) of $X$ with parameters $\mathbf a = (a_i)$, and write $\mathbf Y = \mathbf a \circ X$.
\end{definition}

We have chosen to define the multinomial distribution for $|\mathbf p| \leq 1$, so $|\mathbf Y| \leq n$; we can think of this as marking each ball colour $i$ with probability $p_i$ or discarding the ball with probability $1 - |\mathbf p|$, so \emph{at most} $n$ balls remain. An alternative, stricter convention could require $|\mathbf p| = 1$, so $|\mathbf Y| = n$; we could think of this as every ball being coloured and none being discarded, so all $n$ balls remain. We have preferred the former convention, since setting the number of colours to be $c = 1$ in the multinomial distribution then exactly recovers the binomial distribution, and setting $c = 1$ in multinomial marking exactly recovers binomial thinning. While this is convenient, however, there is no real loss from just thinking of the $|\mathbf p| = 1$ convention.

We note immediately one connection between marking and thinning: the marginal distribution of $i$th coordinate $Y_i$ of a marking $\mathbf Y = \mathbf a \circ X$ is the same as the distribution of a thinning $a_i \circ X$.  The extra interest is in that the coordinates of $\mathbf a \circ X$ are not (in general) independent.

\subsection{Examples} \label{sec:ex}

The following are examples of markings of famous probability distributions. While these are sometimes not too difficult to prove directly, it is usually more convenient to use a generating function approach, so we omit direct proofs for now, and give generating function proofs in Example \ref{ex:genfun}.


\begin{example} \label{ex:bin}
Let $X \sim \Bin(n, p)$ be a binomial distribution. Then  the marking $\mathbf a \circ X \sim \Multi(n, p\mathbf a)$ is a multinomial distribution.
\end{example}

\begin{example} \label{ex:po}
Let $X \sim \Po(\lambda)$ be a Poisson distribution. Then  the marking $\mathbf Y =\mathbf a \circ X \sim \Po(\lambda \mathbf a)$ is a multivariate Poisson distribution, in that each $Y_i \sim \Po(\lambda a_i)$ is a Poisson distribution and all the $Y_i$s are independent.
\end{example}

\begin{example} \label{ex:negbin}
Let $X \sim \NegBin(n, q)$ be a negative binomial distribution with failure probability $q$, so with probability mass function
\[ p_X(x) = \binom{n+x-1}{x} q^x \,(1-q)^n . \]
Then  the marking $\mathbf Y =\mathbf a \circ X \sim \NegMulti(n, \mathbf q)$ is a negative multinomial distribution with failure probabilities $\mathbf q =(q_i)$, so with probability mass function
\[ p_{\mathbf Y}(\mathbf y) = \binom{n+|\mathbf y| -1}{\mathbf y} \prod_{i=1}^c q_i^{y_i} \,\big(1-|\mathbf q|\big)^n  , \]
and the relationship between the new failure probabilities $q_i$ and the original failure probability $q$ is
\[ \frac{q_i}{1 - q_i} = a_i \, \frac{q}{1-q} . \]
\end{example}

\begin{example} \label{ex:herm}
Let $X \sim \Herm(\mu, \sigma^2)$ be an Hermite distribution \cite{kemp1,kemp2}, so $X = U + 2V$, where $U \sim \Po(\alpha)$ and $V \sim \Po(\beta)$ are independent, and the distribution is parametrised by the mean $\mu = \alpha + 2\beta$ and the dispersion $\sigma^2 = 2\beta$. Then the marking $\mathbf Y = \mathbf a \circ X\sim \Herm(\boldsymbol\mu, \mathsf\Sigma)$ is a multivariate Hermite distribution \cite{steyn}, so \[ Y_i = U_i + \sum_{j=1}^c V_{ij} + \sum_{k=1}^c V_{ki} , \] where $U_i \sim \Po(\alpha_i)$ and $V_{ij} \sim \Po(\beta_{ij})$ are all independent, and the distribution is parametrised by the means $\mu_i = \alpha_i + \sum_{j} \beta_{ij} + \sum_{k} \beta_{ki}$, and the dispersion--covariance matrix $\mathsf\Sigma = (\sigma_{ij})$, where the dispersions are $\sigma_{ii} = 2\beta_{ii}$ and covariances are $\sigma_{ij} = \beta_{ij} + \beta_{ji}$ for $j \neq i$. The relationships between the original parameters $\mu$, $\sigma^2$ and the new parameters $\mu_i$, $\sigma_{ij}$ are $\mu_i = \mu a_i$, $\sigma_{ij} = \sigma^2 a_i a_j$, or $\boldsymbol \mu = \mu \mathbf a$, $\mathsf\Sigma = \sigma^2 \mathbf{a} \mathbf{a}^\top$.
\end{example}

\subsection{Properties} \label{sec:prop}

We now discuss some properties of marked distributions. To do this, we will need a little more notation. The $k$th factorial moment is
\[ \operatorname{\mathbb E} X^{\ff{k}} = \operatorname{\mathbb E} X(X-1) \cdots (X-k+1) , \]
and the multivariate generalisation is
\[ \operatorname{\mathbb E} \mathbf Y^{\ff{\mathbf k}} = \mathbb E \prod_{i=1}^c Y_i^{\ff{k_i}} . \]
Our results will be much cleaner if we measure the spread of a random variable not by the variance $\Var(X) = \mathbb EX^2 - (\mathbb EX)^2$, but rather by the dispersion \cite{JK}
\[ \Disp(X) = \mathbb EX^{\ff{2}} - (\mathbb EX)^2 = \Var(X) - \mathbb EX .\]
The covariance is the usual definition $\Cov(X,Y) = \mathbb EXY - (\mathbb EX)(\mathbb EY)$.
The factorial moment generating function (FMGF) \cite[Section 1.2.7]{UDD} is $\Phi_X(t) = \mathbb E(1+t)^X$ and the multivariate FMGF is
\[ \Phi_{\mathbf Y}(\mathbf t) = \mathbb E \prod_{i=1}^c (1 + t_i)^{Y_i} ; \]
these are related to the more common (but here less convenient) probability generating function by $\Phi_X(t) = G_X(1 + t)$ and $\Phi_{\mathbf Y}(\mathbf t) = G_{\mathbf Y}(\mathbf 1 + \mathbf t)$, where $\mathbf 1$ is the all-$1$s vector. We write the scalar product as
$\mathbf a ^\top \mathbf x = a_1x_2 + \cdots a_c x_c$.

\begin{theorem} \label{th:prop}
Let $X$ be a random variable taking values in $\mathbb N$, $\mathbf a \in [0,1]^c$ with $|\mathbf a| \leq 1$, and let $\mathbf Y = \mathbf a \circ X$ be the multinomial marking of $X$ with parameters $\mathbf a$. Then we have the following.
\begin{enumerate}
\item $\mathbb EY_i = a_i \, \mathbb EX$.
\item $\Disp(Y_i) = a_i^2 \Disp(X)$.
\item $\Cov(Y_i, Y_j) = a_i a_j \Disp(X)$ for $i \neq j$.
\item ${\displaystyle \operatorname{\mathbb E} \mathbf Y^{\ff{\mathbf k}} = \left(\prod_{i=1}^d a_i^{k_i}\!\right) \operatorname{\mathbb E} X^{\ff{| \mathbf k|}}}$.
\item $\Phi_{\mathbf Y}(\mathbf t) = \Phi_X\big(\mathbf a ^\top \mathbf t\big)$.
\end{enumerate}
\end{theorem}

Note how these properties generalise standard properties of binomial thinning \cite[Section 2.1]{aldridge}: If $Y = a \circ X$, then $\mathbb EY = a \, \mathbb EX$, $\Disp(Y) = a^2 \Disp(X)$, $\operatorname{\mathbb E} Y^{\ff{k}} = a^k \operatorname{\mathbb E} X^{\ff{k}}$, and $\Phi_Y(t) = \Phi_X(at)$.

Note also that our use of the FMGF simplifies part 5 compared to using the more standard probability generating function, which would require the more awkward
\[ G_\mathbf Y(\mathbf t) = G_X \big(1 - |\mathbf a|+ \mathbf a^\top\mathbf t \big) . \]

\begin{proof}[Proof of Theorem \ref{th:prop}]
The key is part 5. We have
\[ \Phi_{\mathbf Y} (\mathbf t) = \mathbb E\prod_{i=1}^c (1 + t_i)^{Y_i} = \mathbb E\,\mathbb E\left( \left.\prod_{i=1}^c (1 + t_i)^{Y_i} \;\right|\; X \right) . \]
The inner conditional expectation is the FMGF of a multinomial $\Multi(X, \mathbf a)$ distribution, which is $(1 + \mathbf a^\top \mathbf t)^X$; hence,
\[ \Phi_{\mathbf Y} (\mathbf t) = \mathbb E \big(1 + \mathbf a^\top \mathbf t\big)^X = \Phi_X\big(\mathbf a^\top \mathbf t\big). \]

Part 1 to 4 follow from this, since the generating function forms of the FMGFs
\[ \Phi_X(t) = \sum_{k=0}^\infty \mathbb EX^{\ff k} \; \frac{t^k}{k!} \qquad \Phi_\mathbf Y(\mathbf t) = \sum_{k_1, \dots, k_c} \operatorname{\mathbb E} \mathbf X^{\ff{\mathbf{k}}} \; \frac{t_1^{k_1}\cdots t_c^{k_c} }{k_1! \cdots k_c!} \]
allow one to find the factorial moments by differentiating the FMGF an appropriate number of times and setting $t = 0$.
For part 1, we have
\[  \mathbb EY_i = \frac{\partial}{\partial t_i} \Phi_\mathbf Y(\mathbf t) \Big|_{\mathbf t = \mathbf 0} = \frac{\partial}{\partial t_i} \Phi_X \big(\mathbf a^\top \mathbf t\big) \Big|_{\mathbf t = \mathbf 0} = a_i \, \frac{\mathrm d}{\mathrm d t} \Phi_X(t)\Big|_{t = 0} = a_i \, \mathbb EX . \]
For parts 2 and 3, the same method with double differentiation gives
\begin{align*}
\mathbb EY_i(Y_i-1) =  \frac{\partial^2}{\partial t_i^2} \Phi_\mathbf Y(\mathbf t) \Big|_{\mathbf t = \mathbf 0} &= a_i^2 \; \mathbb EX(X-1) \\
\mathbb EY_iY_j =  \frac{\partial^2}{\partial t_i\, \partial t_j} \Phi_\mathbf Y(\mathbf t) \Big|_{\mathbf t = \mathbf 0} &= a_ia_j \, \mathbb EX(X-1) ,
\end{align*}
and the results for the dispersions and covariances follow. Part 4 is the same, but with multiple differentiation according to the vector $\mathbf k$.
\end{proof}

\begin{example} \label{ex:genfun}
We can now prove the examples from Section \ref{sec:ex} very cleanly using the factorial moment generating function (FMGF).
\begin{itemize}
  \item The FMGF of $X \sim \Bin(n,p)$ is $\Phi_X(t) = (1+pt)^n$, while the FMGF of $\mathbf Y \sim \Multi(n, \mathbf p)$ is $\Phi_{\mathbf Y}(\mathbf t) = (1 + \mathbf p ^\top\mathbf t)^n$. This proves Example \ref{ex:bin}.
  \item The FMGF of $X \sim \Po(\lambda)$ is $\Phi_X(t) = \mathrm{e}^{\lambda t}$, while the FMGF of $\mathbf Y \sim \Po(\boldsymbol\lambda)$ is $\Phi_{\mathbf Y}(\mathbf t) = \mathrm{e}^{\boldsymbol\lambda{\!}^\top  \mathbf t}$. This proves Example \ref{ex:po}.
  \item The FMGFs of $X \sim \NegBin(n,q)$ and $\mathbf Y \sim \NegMulti(n, \mathbf q)$ are
  \[ \Phi_X(t) = \left(1+\frac{q}{1-q}\,t\right)^{\!\!-n} \qquad \Phi_{\mathbf Y}(\mathbf t) = \left(1+\sum_{i=1}^d \frac{q_i}{1-q_i}\,t_i\right)^{\!\!-n} . \]
  This proves Example \ref{ex:negbin}.
  \item The FMGF of $X \sim \Herm(\mu, \sigma^2)$ is $\Phi_X(t) = \exp(\mu t + \tfrac12 \sigma^2 t^2)$, while the FMGF of $\mathbf Y \sim \Herm(\boldsymbol\mu, \mathsf\Sigma)$ is $\Phi_{\mathbf Y}(\mathbf t) = \exp(\boldsymbol\mu^\top \mathbf t + \tfrac12 \mathbf t^\top \mathsf\Sigma \mathbf t)$. This proves Example \ref{ex:herm}.
\end{itemize}
\end{example}

We have seen from Example \ref{ex:po} that if $X \sim \Po(\lambda)$ is a Poisson distribution then the coordinates of the marking $\mathbf a \circ X$ are independent. In fact, this is the only case in which such independence occurs.

\begin{theorem}
Let $X$ be a random variable taking values in $\NN$, and let $\mathbf Y = \mathbf a \circ X$ where the marking parameters $\mathbf a = (a_i) \in [0,1]^c$ satisfy $a_i > 0$ for all $i$. Then the coordinates $Y_i$ of $\mathbf Y$ are independent if and only if $X$ is a Poisson distribution.
\end{theorem}

\begin{proof}
The `if' part is Example \ref{ex:po}; we have to prove the `only if'.

From Theorem \ref{th:prop} part 5, the FMGF of the marking $\mathbf Y$ is $\Phi_{\mathbf Y}(\mathbf t) = \Phi_X(\mathbf a^{\top}\mathbf t)$. We also know that each coordinate $Y_i = a_i \circ X$ is a thinning, so has FMGF $\Phi_{Y_i}(t) = \Phi_X(a_i t)$. So the condition for independence is that
\[ \Phi_{X}\big(\mathbf a^\top \mathbf t\big) = \prod_{i=1}^c \Phi_X(a_i t_i) , \]
which is Cauchy's exponential functional equation.
Since $\Phi$ is continuous and bounded on $(-2,0)$, the only solutions are $\Phi_X(t) = \mathrm{e}^{\lambda t}$. For $\lambda \geq 0$, this is the FMGF of the Poisson distribution $X \sim \Po(\lambda)$. For $\lambda < 0$ this does not define a valid FMGF, since (for example) the first derivative is negative at $t = -1$.
\end{proof}

\begin{example}
We now return to the example of red and blue balls from the beginning of this article. We have $X$ balls that are independently painted red with probability $r$ and blue with probability $b$. In the language of this article, the numbers of red and blue balls $(R, B)$ are a marking of $X$; that is,
\[ \begin{pmatrix} R \\ B \end{pmatrix} = \begin{pmatrix} r \\ b \end{pmatrix} \circ X . \]
By Theorem \ref{th:prop}, we have the following:
\begin{align*}
    \mathbb ER = r \, \mathbb EX \quad & \quad \qquad \!\!\!
    \mathbb EB = b \, \mathbb EX \\
    \Disp(R) &= r^2  \Disp(X) \\
    \Disp(B) &= b^2  \Disp(X) \\
    \Cov(R,B) &= rb  \Disp(X) .
\end{align*}
We pause to note how much more convenient these last three results are with our use of the dispersion $\Disp(X) = \Var(X) - \mathbb EX$ than had we used the variance:
\begin{align*}
\Var(R) &= r^2\Var(X) + r(1-r)\,\mathbb E X \\
\Var(B) &= b^2\Var(X) + b(1-b)\,\mathbb E X \\
\Cov(R, B) &= rb\Var(X) - rb\,\mathbb E X .
\end{align*}

We posed the question of whether $R$ and $B$ are positively correlated, negatively correlated, or uncorrelated. Since the correlation has the same sign as the covariance (and since we assumed $r, b \neq 0$), we immediately see the answer is that the correlation has the same sign as the dispersion $\Disp(X)$ of the number of balls:
\begin{itemize}
\item If the number of balls $X$ is underdispersed, in that its variance is smaller than its mean, or $\Disp(X) < 0$, then $R$ and $B$ are negatively correlated. This is because the possible number of balls is tightly constrained, so a ball painted red is a lost opportunity to paint one of the limited number of balls blue, and vice versa.
\item If the number of balls $X$ is overdispersed, in that its variance is larger than its mean, or $\Disp(X) > 0$, then $R$ and $B$ are positively correlated. This is because there is a wide range in the possible number of balls, so lots of red balls is a sign there were likely lots of unpainted balls to start with, so there will likely be lots of blue balls too, and the opposite if there are few red balls.
\item If the number of balls $X$ is equidispersed, in that its variance is equal to its mean, or $\Disp(X) = 0$, then $R$ and $B$ are uncorrelated. This includes the special case of having a Poisson number of balls $X \sim \Po(\lambda)$, which has mean and variance equal to $\lambda$, so dispersion $0$. Here, $R \sim \Po(r\lambda)$ and $B \sim \Po(b\lambda)$ are not just uncorrelated but are independent. For other equidispersed distributions, $R$ and $B$ will be uncorrelated but dependent.
\end{itemize}
\end{example}

\section{Multivariate re-marking} \label{sec:remarking}

Marking is always applied to a univariate discrete random variable $X$, representing a single collection of items. But we can also apply this idea to a collection of items $\mathbf X = (X_1, X_2, \dots, X_d)$, representing $X_1$ items already marked $1$, $X_2$ items already marked $2$, and so on. These items can then be `re-marked' -- that is, receive new marks $1, 2, \dots, c$. An item originally marked $j$ is re-marked $1$ with probability $a_{1j}$, re-marked $2$ with probability $a_{2j}$, and so on. The re-marking probabilities are convenient to keep track of in a matrix $\mathsf A = (a_{ij})$, where $i = 1, 2, \dots, c$ and $j = 1, 2, \dots, d$, so $\mathsf A$ is a $c \times d$ matrix whose columns $\mathbf a_j \in [0,1]^c$ each sum to at most $1$. Here, $c$ is the number of colourings in the new re-marking and $d$ the number of colours in the original marking; it can be simpler to just consider the case $c = d$. The re-marking will be denoted $\mathbf Y = \mathsf A \circ \mathbf X$.

\begin{definition}
Let $\mathbf X$ be a random variable taking values in $\mathbb N^d$. Let $\mathsf A = (a_{ij}) \in [0,1]^{c \times d}$ be a matrix where each column $\mathbf a_j \in [0,1]^c$ satisfies $|\mathbf a_j| \leq 1$. Define random variables $\mathbf Z^{(1)}, \dots, \mathbf Z^{(d)}$ as follows: the conditional distribution of $\mathbf Z^{(j)}$ given $\mathbf X$ is multinomial $\mathbf Z^{(j)} \sim \Multi(X_j, \mathbf a_j)$, and the $\mathbf Z^{(j)}$s are conditionally independent given $\mathbf X$. Write $\mathbf Y = \sum_{j=1}^d \mathbf Z^{(j)}$. Then $\mathbf Y$ is the \emph{multinomial re-marking} of $\mathbf X$ with parameters $\mathsf A$, written $\mathbf Y = \mathsf A \circ \mathbf X$.
\end{definition}

This definition is a little ungainly, but the idea is that $\mathbf Z^{(1)}$ is the re-marking of all the items originally marked $1$, $\mathbf Z^{(2)}$ the re-marking of the items originally marked $2$, and so on; and the total numbers of re-marked items are the sums of these summed over all the original markings.

The binomial thinning $a \circ X$ of a discrete random variable $X$ maintains many similar properties to a scaling $aX$ of a real-valued random variable $X$ \cite[Section 2.1]{aldridge}. Similarly, we will see that the re-marking $\mathsf A \circ \mathbf X$ of a discrete random vector $\mathbf X$ maintains many similar properties to a linear transformation of a real-valued random vector $\mathbf X$, as produced by a matrix multiplication $\mathsf A \mathbf X$.

We give just two examples of re-marking; these are easily proven using the FMGF using Theorem \ref{th:prop2} part 3, which follows in a moment.

\begin{example}
Let $\mathbf X \sim \Po(\boldsymbol\lambda)$ be a multivariate Poisson distribution, meaning the $X_i \sim \Po(\lambda_i)$ are independent. Then  the re-marking $\mathsf A \circ \mathbf X \sim \Po(\mathsf A \boldsymbol\lambda)$ is also a multivariate Poisson distribution.
\end{example}

\begin{example}
Let $\mathbf{X} \sim \Herm(\boldsymbol\mu, \mathsf \Sigma)$ be a multivariate Hermite distribution, as defined in Example \ref{ex:herm}. Then  the re-marking $\mathsf A \circ \mathbf X \sim \Herm(\mathsf A \boldsymbol\mu, \mathsf A \mathsf \Sigma \mathsf A^\top)$ is also a multivariate Hermite distribution.

This is the most interesting example in this article, since the result can be considered a discrete equivalent to the result that if $\mathbf X \sim \NN(\boldsymbol\mu, \mathsf \Sigma)$ is a multivariate normal distribution, then the linear transformation $\mathsf A \mathbf X \sim \NN(\mathsf A \boldsymbol\mu, \mathsf A \mathsf \Sigma \mathsf A^\top)$ is also a multivariate normal distribution: note that the parameters have changed in exactly the same way. This is further evidence for the contention that the Hermite distribution acts like a `discrete normal distribution'. See \cite[Section 3.3]{aldridge} for more on the Hermite distribution as a discrete equivalent to the normal distribution.
\end{example}

In the following theorem, we write $\DispM(\mathbf Y)$ for the dispersion--covariance matrix, whose diagonal entries are the dispersions $\Disp(X_i) = \Var(X_i) - \mathbb E X_i$ and off-diagonal entries are the covariances $\Cov(X_i, X_j)$.

\begin{theorem} \label{th:prop2}
Let $\mathbf X$ be a random variable taking values in $\mathbb N^c$, $\mathsf A \in [0,1]^{c\times d}$ with columns $|\mathbf a_i| \leq 1$, and let $\mathbf Y = \mathsf A \circ X$ be the multinomial re-marking of $\mathbf X$ with parameters $\mathsf A$. Then we have the following.
\begin{enumerate}
\item  $\mathbb E \mathbf Y = \mathsf A \,\mathbb E \mathbf X$.
\item $\DispM( \mathbf Y) = \mathsf A \DispM(\mathbf X)\, \mathsf A^\top$.
\item $\Phi_{\mathbf Y}(\mathbf t) = \Phi_{\mathbf X}(\mathsf A^\top \mathbf t)$.
\end{enumerate}
\end{theorem}

These properties act as discrete equivalents of standard results of the linear transformation $\mathbf Y = \mathsf A \mathbf X$ of a real-valued random variable $\mathbf X$: the expectation is $\mathbb E \mathbf Y = \mathsf A \,\mathbb E \mathbf X$, the variance--covariance matrix is $\VarM( \mathbf Y) = \mathsf A \VarM(\mathbf X)\, \mathsf A^\top$, and the multivariate moment generating function is $M_{\mathbf Y}(\mathbf t) = M_{\mathbf X}(\mathsf A^\top \mathbf t)$

\begin{proof}
The proof is similar to that of Theorem \ref{th:prop}. We start with part 3. We have
\begin{align*}
\Phi_\mathbf{Y}(\mathbf t) = \mathbb E \prod_{i=1}^c (1+t_i)^{Y_i}
  &= \mathbb E \prod_{i=1}^c (1+t_i)^{Z^{(1)}_i + \cdots + Z^{(d)}_i} \\
  &= \mathbb E \prod_{i=1}^c \prod_{j=1}^d (1+t_i)^{Z^{(j)}_i} \\
  &= \mathbb E \prod_{j=1}^d \mathbb E\left( \left.\prod_{i=1}^c (1+t_i)^{Z^{(j)}_i} \;\right|\; \mathbf X \right)  \\
  &= \mathbb E \prod_{j=1}^d \big(1 + \mathbf a_j^\top \mathbf t\big)^{X_j} ,
\end{align*}
where we have used conditional independence of the $\mathbf Z^{(i)}$ given $\mathbf X$. But this is precisely the FMGF $\Phi_{\mathbf X}(\mathbf s)$ of $\mathbf X$ evaluated at $\mathbf s$ where $s_j = \mathbf a_j^\top \mathbf t$, so $\mathbf s = \mathsf A^\top \mathbf t$.

We can now use this result to prove parts 1 and 2. For part 1 we have
\[  \mathbb EY_i = \frac{\partial}{\partial t_i} \Phi_\mathbf Y(\mathbf t) \Big|_{\mathbf t = \mathbf 0} = \frac{\partial}{\partial t_i} \Phi_\mathbf X \big(\mathbf A^\top \mathbf t\big) \Big|_{\mathbf t = \mathbf 0} = \sum_{j=1}^d a_{ij} \, \frac{\partial}{\partial t_j} \Phi_X(\mathbf t)\Big|_{\mathbf t = \mathbf 0} = \sum_{j=d}^c a_{ij} \, \mathbb EX_j . \]
For part 2, the same method with double differentiation shows
\begin{align*}
\mathbb EY_i(Y_i-1) &= \sum_{k \neq l} a_{ik} a_{il} \, \mathbb EX_k X_l
+ \sum_{k = 1}^d a_{ik}^2 \, \mathbb EX_k (X_k - 1) \\
\mathbb EY_iY_j &= \sum_{k \neq l} a_{ik} a_{jl} \, \mathbb EX_k X_l
+ \sum_{k = 1}^d a_{ik} a_{jk} \, \mathbb EX_k (X_k - 1) ,
\end{align*}
and the results for the dispersions and covariances follow.
\end{proof}

\bibliographystyle{abbrvurl}
\bibliography{bibliography}

\end{document}